\documentclass[12pt,a4paper]{article}
\usepackage[T1]{fontenc}
\usepackage{amsmath}
\usepackage{amsthm}
\usepackage{amsfonts}
\usepackage{amssymb}
\usepackage{setspace}
\usepackage{tikz-cd}
\usepackage{float}
\usepackage{graphicx}
\usepackage{pgf,tikz,pgfplots}
\usepackage{mathrsfs}
\usetikzlibrary{arrows}
\usepackage{dsfont}

\newtheorem{theorem}{Theorem}[section]
\newtheorem{definition}[theorem]{Definition}
\newtheorem{lemma}[theorem]{Lemma}
\newtheorem{proposition}[theorem]{Proposition}
\newtheorem{remark}[theorem]{Remark}

\newtheorem{corollary}[theorem]{Corollary}
\newtheorem{fact}[theorem]{Fact}

\renewcommand{\leq}{\leqslant}
\renewcommand{\geq}{\geqslant}

\makeatletter
\renewcommand\section{\@startsection {section}{1}{\z@}%
                               {-3.5ex \@plus -1ex \@minus -.2ex}%
                               {2.3ex \@plus.2ex}%
                               {\centering\normalfont\Large\bfseries}}
\renewcommand\subsection{\@startsection{subsection}{2}{\z@}%
                                 {-3.25ex\@plus -1ex \@minus -.2ex}%
                                 {1.5ex \@plus .2ex}%
                                 {\centering\normalfont\large\bfseries}}
\makeatother

\title{SRB measures for hyperbolic attractor in low regularity}
\author{Houssam BOUKHECHAM}
\date{25-05-2022}

\begin{document}
\maketitle
\begin{abstract}
    We consider a $C^1$ hyperbolic attractor, and prove the existence of a physic measure provided that the differential satisfies some summability condition which is weaker than Hölder continuity.
\end{abstract}

\section{Introduction}



Let $M$ be a compact Riemannian manifold, $f:M\to M$ be a $C^1$ diffeomorphism, and $\Lambda$ a closed $f$-invariant hyperbolic attractor (see section 2.1). Such map have a lot of invariant measures, however there is an interesting class called physique measures. If $f$ is $C^{1+\alpha},$ it is a classical result that $f$ has a unique physical measure (see \cite{Young2002SRBsurvey}). In this paper, we will be interested in the existence of a physical measure in a finer regularity. If $\omega:\mathbb{R}_+\to\mathbb{R}_+$ is a modulus of continuity, we say that $f$ is $C^{1,\omega}$ if $f$ is $C^1$ and the modulus of continuity of $df$ is a multiple of $\omega,$ i.e there is $C>0$ such that
\begin{equation*}
    \|df_x-df_y\|\leq C\omega\big(d(x,y)\big), \forall x,y\in M.
\end{equation*}
We say that a modulus $\omega$ is Dini summable if
\begin{equation*}
    \int_0^1\frac{\omega(t)}{t}~dt<+\infty.
\end{equation*}
Our main results extends the work of Fan and Jiang \cite{FanJiang} from expanding maps to hyperbolic diffeomorphisms.
\begin{theorem}\label{Physic}
If $\Lambda$ is a hyperbolic attractor of a $C^{1,\omega}$ diffeomorphism $f$, and $\omega$ is Dini summable, then $f$ has a unique physical measure.
\end{theorem}

To prove this theorem, first, we study the modulus of continuity of the unstable distribution $E^u,$ and prove that

\begin{theorem}\label{main theorem}
Let $f:\Lambda\to\Lambda$ be a $C^{1,\omega}$ hyperbolic map, where $\omega$ is Dini summable modulus of continuity, then the unstable distribution has a Dini summable modulus of continuity.
\end{theorem}
Given this theorem, we deduce that the geometric potential
\begin{equation*}
\phi^{(u)}=-\log J^uf=-\log\det df_{|E^u}   \end{equation*} has a Dini summable modulus of continuity. Finally, using Markov partitions \cite{bowen1975equilibrium}, $f:\Omega(f)\to \Omega(f)$ is semiconjugated to a subshift of finite type. More precisely, there are $(\Sigma_A,\sigma)$ and a surjective Hölder map $\pi:\Sigma_A\to\Lambda$ such that $\pi\circ\sigma=f\circ\pi$, so if we take a potential with Dini summable modulus $\phi:\Lambda\to\mathbb{R},$ then $\pi\circ\sigma$ has a Dini summable modulus. To get an equilibrium state for $(\Sigma_A,\sigma,\pi\circ\phi)$ one can after several lemmas consider only the one-sided shift $(\Sigma_A^+,\sigma,\Tilde{\phi}),$ ( where $\Tilde{\phi}$ is a potential depending only on the future, and cohomologous to $\pi\circ\phi$ \cite{Polli1990Zeta}) which is an expanding map, then we apply the adapted Ruelle-Perron-Frobenius theorem \cite{FanJiang} to get an equilibrium measure for $(\sigma,\pi\circ\phi).$ We push this measure by $\pi$ to get an equilibrium measure for $\big(f_{\Omega(f)},\phi\big).$ 

\begin{corollary}
Let $f:\Lambda\to\Lambda$ be a $C^{1,\omega}$ hyperbolic map, where $\omega$ is a summable modulus of continuity, then the geometric potential $\phi^{(u)}$ has a unique 
equilibrium state $\mu_{\phi^{(u)}},$ in particular $\mu_{\phi^{(u)}}$ is ergodic. Furthermore, if $f$ is topologically mixing then the measure $\mu_{\phi^{(u)}}$ is Bernoulli.  
\end{corollary}

The article is organized as follows. In section 2, we will recall uniform hyperbolicity and some classical tools; modulus of continuity, and equivalent formulation of the Dini summability condition. In section 3 we will recall some classical and motivational examples. In section 4 we give the proof of the regularity of the potential and proceed in the same way as in \cite{bowen1975equilibrium} to prove the existence of equilibrium measure. In section 5, we prove theorem \ref{Physic} by adapting some volume lemmas given in \cite{bowen1975equilibrium}.

\section{Preliminaries and notations}
\subsection{Uniform hyperbolicity}

Let $U$ be an open subset of $M$ and $f:U\to M$ a $C^1$ diffeomorphism. An invariant set $\Lambda\subset U$ is called hyperbolic if there are some $C>0$ and $\lambda\in (0,1)$ such that for all $x\in\Lambda$ we have a splitting of $T_xM=E^u_x\oplus E^s_x$ which is $f$ invariant, i.e $df_x(E^u_x)=E^u_{fx}$ and $df_x(E^s_x)=E^s_{fx}$ and such that
\begin{align}
    \|df^n(v)\|&\leq C\lambda^n\|v\|, \forall n\in\mathbb{N},~v\in E^s_x,
    \\\|df^{-n}(v)\|&\leq C\lambda^n\|v\|, \forall n\in\mathbb{N},~v\in E^u_x.
\end{align}

In the definition, we didn't assume any continuity on $E^s$ and $E^u$. In fact it is not hard to prove the continuity of $E^u$ and $E^s$ starting from the given definition. One can also assume that $C=1$ by considering another equivalent Riemannian metric on $M$, and taking $\lambda'\in(\lambda,1)$ (see Proposition 5.2.2 of \cite{BrinStuck}).

Some example of uniformly hyperbolic maps are: Arnold cat map, the Horseshoe, toral hyperbolic linear automorphism, the Smale solenoid.

A classical approach to deal with uniform hyperbolic maps, is to consider the space of continuous (resp. bounded) sections $\sigma:\Lambda\to T\Lambda,$ which is a Banach space, and once we have the first definition, we can write this Banach space as the direct sum of two closed subspaces, corresponding to sections with value on $E^u$ or $E^s$. Once we do this, we have a natural linear action of $f$ on that Banach space which preserves the closed subspaces. This approach helps us prove a lot of result like shadowing lemma, local stability, etc.

In general it is hard to check uniform hyperbolicity using this definition (for instance we don't know $E^u$ and $E^s$), to deal with this difficulty we study cones instead of linear subspaces.

\subsection{Hyperbolicity via cone techniques}
Let $x\in M$ and $E$ a linear subspace of $T_xM,$ define the cone centered at $E$ by
\begin{equation*}
    K_\alpha^E(x)=\left\lbrace v\in T_xM : \|v_2\|\leq \alpha\|v_1\| \text{ where } v=v_1+v_2 \text{ and } v_1\in E, v_2\in E^{\perp} \right\rbrace.
\end{equation*}
For a hyperbolic map $f,$ $K^{E^u}_\alpha$ (resp $K^{E^s}_\alpha$) is called unstable (resp stable) cone field. We say that it has a small angle if $\alpha$ is small.

A cone field $K$ on $M$ is said to be invariant by $f$ if for all $x\in M$
\begin{equation*}
    df_x\big(K(x)\big)\subset int\big(K(fx)\big) \cup \lbrace 0\rbrace.
\end{equation*}

\begin{proposition}(Proposition 5.4.3 \cite{BrinStuck})
Let $\Lambda$ be a compact invariant set of $f:U\to M.$ Suppose that there is $\alpha>0$ and for every $x\in\Lambda$ there are continuous subspaces $\Tilde{E}^s$ and $\Tilde{E}^u(x)$ such that $\Tilde{E}^s(x)\oplus\Tilde{E}^u(x)=T_xM,$ and the cone $K_\alpha^{\Tilde{E}^u}(x)$ and $K_\alpha^{\Tilde{E}^s}(x)$ are $f$ invariant and $\|df_xv\|< \|v\|$ for non zero $v\in K_\alpha^{\Tilde{E}^s}(x),$ and $\|df_x^{-1}v\|<\|v\|$ for non-zero $v\in K_\alpha^{\Tilde{E}^u}(x).$
Then $\Lambda$ is a hyperbolic set of $f$.
\end{proposition}

\subsection{Local stable and unstable manifolds:}
Define the local stable and unstable manifolds for $x\in\Lambda$ by

\begin{align*}
W^s_\epsilon(x)&=\lbrace y\in M~|~ d(f^kx,f^ky)\leq \epsilon, \forall k\geq 0\rbrace,\\
W^u_\epsilon(x)&=\lbrace y\in M~|~ d(f^{-k}x,f^{-k}y)\leq \epsilon, \forall k\geq 0\rbrace.
\end{align*}
The definition of stable and unstable manifolds is dynamic, and the theorem of Perron-Hadamard proves that $W^u_{\epsilon}$ and $W^s_{\epsilon}$  are sub-manifolds for $\epsilon$ small enough.

Let's recall some classical definitions:\\
\textit{Pseudo-orbit:}
Let $f:X\to X$ a homeomorphism of a metric space, and let $\epsilon>0.$ We say that a sequence $(x_n)_{n\in\mathbb{Z}}$ is a $\epsilon$-pseudo-orbit if $d(fx_n,x_{n+1})<\epsilon.$\\\\
\textit{Shadowing lemma:} 
We say that a homeomorphism $f:X\to X$ have the shadowing property if for all $\epsilon>0$ there is $\delta>0$ such that for all $\epsilon$-pseudo-orbit $(x_n)_{n\in{\mathbb{Z}}}$ there is $y\in X$ such that for all $n\in\mathbb{Z}$ we have 
$d(f^ny,x_n)\leq \delta.$\\
It is known that Anosov diffeomorphisms and Axiom A diffeomorphisms have the shadowing property, in fact they have a stronger property called specification.\\\\
\textit{Expansiveness:}
A homeomorphism $f:X\to X$ is expansive if there is $\epsilon_0>0$ such that for all $x,y\in X$ there is $n\in\mathbb{Z}$ such that $d(f^nx,f^ny)\geq \epsilon_0.$ For instance an isometry is not expansive, and hyperbolic maps are expansive.\\\\
\textit{Wandering set:} $x$ is non-wandering, if for all neighborhood $U$ of $x$ there is $n>0$ such that $U\cap f^nU\neq\emptyset.$ we denote the wandering set of a diffeomorphism $\Omega(f).$\\\\
\textit{Axiom A diffeomorphism:} $f$ is an Axiom A diffeomorphism if $\Omega(f)$ is hyperbolic, and periodic orbits are dense in $\Omega(f).$
\begin{remark}
Local stable and unstable manifolds are used to construct Markov partitions, then shadowing, expansiveness 
and density of periodic orbits are used to prove that an Axiom A diffeomorphism is semiconjugated to a subshift of 
finite type via a Hölder map (see \cite{bowen1975equilibrium}).
\end{remark} \textit{Spectral Decomposition:}(Proposition 3.5 \cite{bowen1975equilibrium}) Let $f$ be an Axiom A diffeomorphism, then one can write $\Omega(f)=\Omega_1\cup\cdots\cup\Omega_s$
where the $\Omega_i$ are pairwise disjoint closed sets (called basic sets) such that
\begin{enumerate}
    \item $f(\Omega_i)=\Omega_i$ and $f_{|\Omega_i}$ is topologically transitive;
    \item $\Omega_i=X_{1,i}\cup\cdots X_{n_i,i}$ with $X_{j,i}$'s pairwise disjoint closed sets, $f(X_{j,i})=X_{j+1,i}$ ($X_{n_j+1,i}=X_{1,i}$) and $f^{n_i}_{|X_{j,i}}$ is topologically mixing.
\end{enumerate}
\begin{remark}
Using the spectral decomposition theorem, one can assume without loss of generality that an Axiom A diffeomorphism is transitive.
\end{remark}
\textbf{Physical measure}:
Let $f:M\to M$ be a continuous map, $X$ a closed $f$-invariant subset of $M$ and $\mu$ a $f$-invariant probability measure with support in $X.$ Define the basin of the measure $\mu$ by:
\begin{equation*}
    B_\mu=\left\lbrace x\in M~ \Bigg| ~\forall g\in C^0(M,\mathbb{R}), \lim\limits_{n\to\infty}\frac{1}{n}\sum\limits_{k=0}^{n-1}g\circ f^k(x)=\int_{\Lambda} g~d\mu\right\rbrace.
\end{equation*}
We say that $\mu$ is physical if $B_\mu$ has positive measure with respect to the Lebesgue measure of $M.$\\\\
\textit{Attractor:}
A basic set $\Omega_s$ is called an attractor if it has a neighborhood $U$ such that $f(U)\subset U.$

\subsection{Grassmanian bundle}

Let $G^q(M)=\bigsqcup\limits_{x\in M}G(q,T_xM)$ be the fiber bundle over $M$ with fiber $G(q,T_xM)$ (the set of subspaces of $T_xM$ of dimension $q)$.
A continuous distribution $E$ is a continuous section of $G^q(M),$ the latter has a Riemannian metric, so one can talk about the modulus of continuity of a distribution or the distance $d_{grass}$ between distributions that comes from the Riemannian structure.



Consider a continuous distribution $E.$ Let $x_0\in M,$ and consider a chart $\psi:U\to\mathbb{R}^m,$ where $U$ is a small neighborhood of $x_0$ such that for all $x\in U,$ $E(x)$ is sufficiently close to $E(x_0)$ and $d\psi_{x_0}E(x_0)=\mathbb{R}^q\times\lbrace 0\rbrace,$ where $q=\dim E.$ Define the distance $d=d_{x_0,\psi,E(x_0)}$ on a small neighborhood $\Tilde{U}$ of $(x_0,E_0)$ in $G^q(M)$ by:
\begin{equation*}
d(F_x^1,F_y^2):=d(x,y)+\|L_{F_x^1}-L_{F_y^2}\|,
\end{equation*}
where $L_{F_x^1}:\mathbb{R}^q\to\mathbb{R}^{m-q}$ (resp $L_{F_y^2}$) is the linear map whose graph is $d\psi_{x}(F_x^1)$ \big(resp $d\psi_{y}(F_y^2)$\big). This distance induces locally the usual topology of $G^q(M).$





\subsection{Pressure and equilibrium measure}

In this part, let $(X,d)$ be a compact metric space, $f:X\to X$ a continuous function, and $\phi:X\to \mathbb{R}$ a potential. Define the dynamical distance $d_n$ for $n\in\mathbb{N}$ and $x,y\in X$ by:
\begin{equation*}
    d_n(x,y)=\sup\limits_{0\leq k<n}d(f^kx,f^ky).
\end{equation*}
We denote by $B_n(x,\epsilon)$ the ball of center $x$ and radius $\epsilon$ with respect to the distance $d_n.$

A set $A$ is called $(n,\epsilon)$-spanning \big(resp. $(n,\epsilon)$-separated\big) if $X=\bigcup\limits_{x\in A}B_n(x,\epsilon)$ (resp. for any $x,y\in A$ we have $B_n(x,\epsilon)\cap B_n(y,\epsilon)=\emptyset).$ 

Consider the two following numbers that depend on $f,\phi,\epsilon$ and $n\geq 1$
\begin{align*}
    Q_n(f,\phi,\epsilon)&=\inf\left\lbrace \sum\limits_{x\in A}e^{(S_nf)(x)} : A \text{ is a } (n,\epsilon)-\text{spanning set for }X\right\rbrace,\\
     P_n(f,\phi,\epsilon)&=\sup\left\lbrace \sum\limits_{x\in A}e^{(S_nf)(x)} : A \text{ is a } (n,\epsilon)-\text{separated set for }X\right\rbrace,
\end{align*}
where $(S_nf)(x)=\sum\limits_{k=0}^{n-1}\phi \circ f^k(x).$ If $\phi$ is continuous, then
$$\lim\limits_{\epsilon\to 0}\limsup\limits_{n\to\infty}\frac{1}{n}\log\big(Q_n(f,\phi,\epsilon)\big)$$ exists, is finite and is equal to  
$$\lim\limits_{\epsilon\to 0}\limsup\limits_{n\to\infty}\frac{1}{n}\log\big(P_n(f,\phi,\epsilon)\big).$$ The limit is called the topological pressure with respect to the potential $\phi$, we denote it by $P(\phi).$

Denote by $\mathcal{M}_f(X)$ the space of $f$-invariant probability measures on $X$.
Let $\mu\in\mathcal{M}_f(X)$, then define the pressure with respect to this measure by:
\begin{equation*}
    P_\mu(\phi)=h_\mu(f)+\int_X \phi~d\mu,
\end{equation*}
where $h_\mu(f)$ is the entropy of $f$ with respect to the measure $\mu.$ The variational principle (see theorem 9.10 in \cite{WaltersIntrErgodic}) gives the following formula:
\begin{equation*}
    P(\phi)=\sup\left\lbrace h_\mu(f)+\int \phi~d\mu : \mu \in \mathcal{M}_f(X)\right\rbrace. 
\end{equation*}
If $\mu\in\mathcal{M}_f(X)$ is such that $P_\mu(\phi)=P(\phi),$ then $\mu$ is called an equilibrium measure for the potential $\phi.$

\subsection{Modulus of continuity}
In this section, we give a formal definition of the summability condition introduced in the abstract. We start by recalling the classical definition of a modulus of continuity.

\begin{definition}
A modulus of continuity is a continuous, increasing and concave map $\omega:\mathbb{R}^+\to\mathbb{R}^+$, such that $\omega(0)=0$.

We say that the modulus $\omega$ is Dini summable if $$\int_0^1\frac{\omega(t)}{t}dt<+\infty.$$

\end{definition}

For instance, for any $\alpha\in (0,1),$ the map $\omega(t)=t^\alpha$ is a modulus of continuity which is Dini summable. We assumed the concavity condition in the definition of a modulus of continuity for technical issues (see next proposition and lemma \eqref{Summable modulus majoration}), but note that any uniformly continuous map admits a concave modulus of continuity (See the end of this section).

The following proposition gives an equivalent condition on a modulus $\omega$ to be Dini summable. 

\begin{proposition}
The following conditions are equivalent:
\begin{itemize}
    \item $\omega$ is Dini summable.
    \item $\forall c\in (0,1)$ and $t\geq 0$, $\sum\limits_{k=0}^{+\infty}\omega(c^kt)<+\infty$
    \item $\exists c\in (0,1)$ and $t> 0$, $\sum\limits_{k=0}^{+\infty}\omega(c^kt)<+\infty$
\end{itemize}
\end{proposition}

\begin{proof}
Since $\omega$ is concave, the map $t\mapsto\frac{\omega(t)}{t}$ is decreasing, hence we have the following inequalities for all $n$ and small $t$:
\begin{equation}
    \sum\limits_{k=0}^{n-1}(c^k-c^{k+1})\frac{\omega(c^kt)}{c^kt}\leq \int_{c^nt}^t\frac{\omega(x)}{x}dx\leq \sum\limits_{k=0}^{n-1}(c^k-c^{k+1})\frac{\omega(c^{k+1}t)}{c^{k+1}t}.
\end{equation}
We deduce the proposition from these inequalities.
\end{proof}

Let $\omega$ be a Dini summable modulus, and for $c\in (0,1)$ define
\begin{equation}\label{Derivative of a modulus}
    \tilde{\omega}_c(t)=\sum\limits_{k=0}^{+\infty}\omega(c^kt), \forall t\geq 0.
\end{equation}

It follows immediately that $\tilde{\omega}_c$ is a modulus of continuity.
\begin{definition}
If $\omega$ is a Dini summable modulus, we denote $\tilde{\omega}$ the modulus defined by $$\tilde{\omega}(t)=\int_0^t\frac{\omega(s)}{s}ds.$$
\end{definition}
Using the previous proposition, the modulus $\tilde{\omega}$ and $\omega_c$ are equivalent for any $c\in (0,1)$ i.e there is $C>1$ such that:
$$C^{-1}\tilde{\omega}_c\leq \tilde{\omega}\leq C\tilde{\omega}_c.$$

\begin{remark}
The Dini summability condition might seem artificial, but it becomes more natural once we see in \cite{FanJiang} how it is used to prove Ruelle theorem for transfer operator.
\end{remark}

\textbf{Examples:}
\begin{itemize}
    \item For $\alpha\in (0,1],$ $\omega(t)=t^{\alpha}$ is Dini summable, because $\tilde{\omega}(t)=\frac{1}{\alpha}\omega(t).$
    \item The modulus $\omega_{\beta\log}(t)=\frac{1}{\big(\log(\frac{1}{t})\big)^{\beta}}$ is Dini summable if and only if $\beta>1$. In this example $\omega$ is defined only for small $t$, then we extend it by an affine map. 
\end{itemize}

\begin{definition}
Let $X,Y$ be two metric spaces, and $\omega$ a modulus of continuity, we say that a map $f:X\to Y$ is $C^{0,\omega}$ if there is a $C>0$ such that:
\begin{equation}\label{modulus of continuity}
    d\big(f(x),f(y)\big)\leq C\omega\big(d(x,y)\big),\quad\forall x,y\in X.
\end{equation}
\begin{itemize}
    \item If $\omega(t)=t$, then $C^{0,\omega}$ is the set of Lipschitz maps.
    \item If $\omega(t)=t^{\alpha}$, where $0<\alpha<1$, then $C^{0,\omega}$ is the set of Hölder maps with exponent $\alpha$.
\end{itemize}
\end{definition}

Given a continuous map $g:M\to M$ of a compact manifold, a natural way to define the modulus of continuity of $g$ would be to take: 
\begin{equation}
    \Tilde{\omega}_g(t)=\sup\limits_{\underset{d(x,y)\leq t}{x,y\in M}}d(gx,gy),
\end{equation}

but $\Tilde{\omega}_g$ is not concave. To get concavity, we take:
\begin{equation}
    \omega_g=\inf\lbrace h~|~ h\text{ continuous, concave and increasing and } h\geq\Tilde{\omega}_g\rbrace. 
\end{equation}
It is clear that $\omega_g$ is a modulus of continuity, and it satisfies the inequality (\ref{modulus of continuity}) with constant $C=1$.

\section{Motivating example: A Horseshoe with Positive Measure}

In this section we will see the importance of the summability condition by comparing proposition 
\ref{Attractor different formulation} and the example given in \cite{bowen1975horseshoe}.


Let $I=[a,b]$ be a closed interval, and $(\alpha_n)$ a sequence of 
positive numbers such that $\sum\limits_{n=0}^{\infty}\alpha_n<l(I).$ Let $\underline{a}=a_1a_2\cdots a_n$ denote a sequence of 
$0$'s and $1$'s of length $n=n(\underline{a}).$ Define $I_{\emptyset}=I, I^*_{\emptyset}=\left[\frac{a+b}{2}-\frac{\alpha_0}{2},
\frac{a+b}{2}+\frac{\alpha_0}{2}\right]$ and $I^*_{\underline{a}}\subset I_{\underline{a}}$ recursively as follows. Let 
$I_{\underline{a}0}$ and $I_{\underline{a}1}$ be the left and right intervals remaining when the interior of $I_{\underline{a}}^*$ is removed from $I_{\underline{a}}.$ Let $I^*_{\underline{a}k} (k=0,1)$ be the closed interval of length $
\frac{\alpha_{n(\underline{a}k)}}{2^{n(\underline{a}k)}}$ and having the same center as $I_{\underline{a}k}.$ The set 
$K_I=\bigcap\limits_{m=0}^{\infty}\bigcup\limits_{n(\underline{a})=m}I_{\underline{a}}$ is the standard Cantor set, and by 
construction the Lebesgue measure of $K_I$ is $l(I)-\sum\limits_{n=0}^{\infty}\alpha_n.$ Let $J$ be another interval, and $(\beta_n)$ be a sequence of positive numbers. We construct $J_{\underline{a}}, J_{\underline{a}}^*$ and $K_J$ as above. Let us take $\beta_n=\frac{1}{(n+10)^2}, \alpha_n=\frac{\beta_{n+1}}{2},$ $\delta_n=2\frac{\beta_{n}}{\beta_{n+1}}-2, I=\left[\frac{\beta_0}{2},1\right]$ and $J=[-1,1].$ For each $\underline{a}$ let $g:I_{\underline{a}}^*\to J_{\underline{a}}^*$ be a $C^1$ diffeomorphism so that
\begin{itemize}
\item[i.] $g'(x)=2$ for $x$ and endpoint of $I_{\underline{a}}^*,$
\item[ii.] $2-\delta_n\leq g'(x)\leq \frac{\beta_n}{\alpha_n}+\delta_n$ for $x\in I_{\underline{a}}^*.$ 
\end{itemize} 
Then $g$ extends from $\bigcup\limits_{\underline{a}}I_{\underline{a}}^*$ to a homeomorphism $g:I\to J;$ $g$ is in fact a $C^1$ diffeomorphism with derivative $2$ at each point of $K_I.$ One defines a diffeomorphism $f$ of the square $S=J\times J$ into $\mathbb{R}^2$ by
\begin{itemize}
\item[i.] $f(x,y)=\big(g(x),g^{-1}(y)\big) \text{ for } (x,y)\in I\times J,$
\item[ii.] $f(x,y)=\big(g(-x),-g^{-1}(y)\big)$ for $(x,y)\in (-I)\times J$ and \\
$f(T)\cap(J\times J)=\emptyset$ where $T=\left(-\frac{\beta_0}{2},\frac{\beta_0}{2}\right)\times J.$
\end{itemize}
The mapping $f$ can be extended to the sphere. Then $\Lambda=\bigcap\limits_{n=-\infty}^{+\infty}f(S)=K_J\times K_J$ has Lebesgue measure $\big(2-\sum\limits_{n=0}^{\infty}\beta_n\big)^2>0.$ The modulus of continuity of $g'$ does not satisfy Dini condition, in particular the modulus of continuity of $df$ does not satisfy Dini condition. This hyperbolic horseshoe has a positive measure but it is not an attractor. Proposition \ref{Attractor different formulation} gives a sufficient condition on the regularity of the hyperbolic map so that we have an equivalence between $\Lambda$ being an attractor and $W^s(\Lambda)$ having a positive Lebesgue measure.

\textbf{Modulus of continuity of $g':$}
Let $\omega$ be the modulus of continuity of $g'$, then we have for all $x\in I,$ and $x_0\in K_I$
\begin{equation}
    |g'(x)-g'(x_0)|\leq \omega\big(d(x,x_0)\big).
\end{equation}
Consider some interval $I^*_{\underline{a}},$ where length$(\underline{a})=n$, since $g$ send $I^*_{\underline{a}}$ of length $\frac{\alpha_n}{2^n}$ to some interval of length $\frac{\beta_n}{2^n}$, there is $x\in I^*_{\underline{a}}$ such that $g(x)\geq 2\frac{\beta_n}{\beta_{n+1}},$ and if we take $x_0$ in the boundary of $I^*_{\underline{a}}$ then $g(x_0)=2$, we deduce that:
$$\delta_n= 2\frac{\beta_n}{\beta_{n+1}}-2 \leq |g(x)-g(x_0)| \leq \omega\big(\frac{1}{2^n}\big), $$
which implies that:
\begin{equation}
    \delta_n\leq \omega(\frac{1}{2^n}), \forall n\in\mathbb{N},
\end{equation}
but it is clear that $\sum\limits_{k\geq 1}\delta_n =+\infty,$ which implies that the modulus does not satisfy Dini condition.

\begin{remark}[Perturbation of hyperbolic toral automorphisms]
We can construct an example of a diffeomorphism $f$ close to the identity, with derivative Dini summable modulus, but which is not Hölder continuous either by perturbing a toral hyperbolic linear automorphism $A$, or by considering a diffeomorphism $g$ close to the identity with derivative having a Dini summable modulus, then considering the map $A^n\circ g\circ A^n$ which is hyperbolic for large $n$ (using the cone field criterion one can prove the hyperbolicity of such a map).
\end{remark}



\section{Proof of theorem \ref{main theorem}}

The strategy to prove theorem \ref{main theorem} is to study the regularity of the unstable distribution restricted to a stable leaf, then determining the regularity of unstable leafs, then deducing by lemma \ref{split regularity} the desired regularity of the unstable distribution.

Let $A:\mathbb{R}^n\to\mathbb{R}^n$ be a hyperbolic linear map, and consider the set $\xi_s$ defined by $$\xi_s=\lbrace E\in G^q\mathbb{R}^n~|~ E\oplus E^s=\mathbb{R}^n\rbrace,$$ where $E^s$ is the stable subspace of $A$. Consider the distance $d_{E^u\oplus E^s}$ on $\xi_s$ given by $$d_{E^u\oplus E^s}(E,F)=\|L_E-L_F\|, ~E,F\in \xi_s$$ where $L_E:E^u\to E^s$ (resp $L_F$) is the unique linear map whose graph is $E$ (resp $F$), and $\parallel\cdot\parallel$ is any operator norm on the space of linear maps from $E^u$ to $E^s$ (See Chapter 4.3 of \cite{ManeSmooth}). 

The following lemma states that this distance is contracted under the action of $A.$
\begin{lemma}\label{Contraction of the action of hyperbolic maps}
We have $d_{E^u\oplus E^s}(AE,AF)\leq\lambda^2d_{E^u\oplus E^s}(E,F),~\forall E,F\in \xi_s,$ where $\lambda\in (0,1)$ is a hyperbolicity constant of $A$.
\end{lemma}

\begin{proof}
We have by definition $L_{AE}=A\circ L_E\circ A^{-1}_{|E^u},$ so
\begin{align*}
    d_{E^u\oplus E^s}(AE,AF)=\|L_{AE}-L_{AF}\|&=\|A(L_E-L_F)A^{-1}_{|E^u}\|\\
    &\leq \parallel A_{|E^s}\parallel\parallel A^{-1}_{|E^u}\parallel\parallel L_E-L_F\parallel\\ &\leq\lambda^2 d_{E^u\oplus E^s}(E,F).
\end{align*}
\end{proof}


The following two lemmas show that the defined distance depends in a bi-Lipschitz way on the reference spaces $E^u$ and $E^s$.

\begin{lemma}\label{Continuity of the nice distance on distribution}
If $B$ is a linear map such that $\|A-B\|\leq \epsilon$, for some small $\epsilon>0$, then we have $d_{E^u\oplus E^s}(BE^u,E^u)\leq \epsilon.$
\end{lemma}
\begin{proof}
Since we take a small $\epsilon,$ $BE^u$ is transverse to $E^s.$ Let us show that $L_{BE^u}=AL_{A^{-1}BE^u}A^{-1}.$ Take $x\in E^u$, then by definition of $L_{BE^u}$ we have $x+L_{BE^u}x\in BE^u,$ so $A^{-1}x+A^{-1}L_{BE^u}x\in A^{-1}BE^u.$ We have also $A^{-1}x+L_{A^{-1}BE^u}A^{-1}x\in A^{-1}BE^u,$ so we deduce that
$$A^{-1}L_{BE^u}=L_{A^{-1}BE^u}A^{-1}_{|E^u}.$$
We have also $L_{E^u}=0,$ hence  $$d_{E^u\oplus E^s}(BE^u,E^u)=\|A_{E^s}L_{A^{-1}BE^u}A^{-1}_{E^u}-0\|\leq\lambda^2\|L_{A^{-1}BE^u}\|,$$ and since $A^{-1}B$ is close to the identity we deduce the lemma.
\end{proof}

\begin{lemma}\label{Benoit approximation}
Let $\epsilon>0$ small, and take $E_0$ (resp $F_0$) a subspace of dim $q$ (resp $m-q$) such that $d(E_0,E^u)\leq\epsilon$ (resp $d(F_0,E^s)\leq \epsilon)$. Then, there is $\delta=\delta(\epsilon),$ such that $\delta(\epsilon)\xrightarrow[\epsilon\to 0]{}1$ and for all $E, E'$ transversal to $E^s$ and $F_0$ and close enough to $E^u$ we have
\begin{equation*}
    \frac{1}{\delta}  d_{E^u\oplus E^s}(E,E')\leq d_{E_0\oplus F_0}(E,E')\leq \delta  d_{E^u\oplus E^s}(E,E').
\end{equation*}
\end{lemma}

\begin{proof} 
Let $M_1$ (resp $M_2$) the linear map from $E^u$ to $E^s$ (resp $E^s$ to $E^u$) whose graph is $E_0$ (resp $F_0$). Since $E_0$ (resp $F_0$) is close to $E^u$ (resp $E^s$) we deduce using the previous lemma that $\|M_i\|<\epsilon,$ for $i=1,2.$ Let $B_\epsilon$ the linear map from $\mathbb{R}^n$ to $\mathbb{R}^n$ defined on $E^u$ as $Id_{E^u}+M_1$ and on $E^s$ as $Id_{E^s}+M_2.$

Let $E$ and $E'$ be two subspaces transverse to $E^s$ and $F_0$, and let $L_E$ (resp $L_E^0$) the linear map from $E^u$ to $E^s$ (resp $E_0$ to $F_0$) whose graph is $E$. 

First, assume that $F_0=E^s$. Let $x\in E^u,$ then we have by definition of $L_E:$
$$B_\epsilon x +(L_Ex-M_1x)= x +M_1x +(L_Ex-M_1x)\in E,$$
and since $L_E^0(B_\epsilon x)$ is the unique vector in $F_0$ such that $B_\epsilon x+L_E^0(B_\epsilon x)\in E,$ we deduce that $$L_E^0B_\epsilon x=L_Ex-M_1x,$$ using the last equality we deduce that
\begin{align*}
    \|L_E^0-L_{E'}^0\|=\|(L_E^0-L_{E'}^0)B_\epsilon B_\epsilon^{-1}\| &\leq \|B_\epsilon^{-1}\|\|L_E^0B_\epsilon-L_{E'}B_\epsilon\|\\
    &\leq \|B_\epsilon^{-1}\|\|L_E-L_{E'}\|.
\end{align*}
Since $B_\epsilon$ is a perturbation of $Id,$ the norm $\|B_\epsilon^{-1}\|$ is close to 1, and goes to 1 when $\epsilon$ goes to 0, which proves the lemma in the special case where $F_0=E^s.$\\
Now assume that $E_0=E^u$, and define $M_\epsilon$ from $E^u$ to $E^u$ $$M_\epsilon x=x+L_Ex-B_{\epsilon}L_Ex=x-M_2L_Ex\in E^u,$$
notice that $M_\epsilon$ does not depend on $E.$
Since $E_0=E^u,$ we have $L_E^0M_\epsilon x=B_\epsilon L_Ex,$ for all $x\in E^u.$ The map $M_\epsilon$ is a small perturbation of $Id_{E^u},$ so it is invertible and we have
\begin{equation}
    L_E^0=B_\epsilon L_EM_\epsilon^{-1},
\end{equation}
we deduce that
\begin{equation}\label{Second step in Benoit approximation}
    \|L_E^0-L_{E'}^0\|=\|B_\epsilon(L_E-L_E^0)M_\epsilon^{-1}\|\leq (1+\epsilon)(1+\epsilon)\|L_E-L_{E'}\|,
\end{equation}
which proves the lemma in the special case where $E_0=E^u.$ Combining the two cases we get the proof in general.

\end{proof}


We will need a few lemmas to get the regularity of the unstable distribution when restricted to a small piece of stable manifold. The strategy is to use the dynamic of $f$ to give a upper bound for the modulus of continuity of $E^u$ when restricted to a local stable leaf.

\begin{lemma}\label{Summable modulus majoration}
Let $f:\Lambda\to\Lambda$ be a $C^{1,\omega}$ hyperbolic map and $\lambda$ the constant given by equation (1) with respect to an adapted norm. Define for $x_0\in\Lambda,$ $n\in\mathbb{N}$ and $c\in(0,\lambda),$ the map $\Omega_n^c:W_{\epsilon}^s(x_0)\times W_{\epsilon}^s(x_0)\to \mathbb{R}^+$
\begin{equation}
    \Omega_n^c(x,y)=\sum\limits_{k=0}^{n-1}c^{n-k}\omega\big(d(f^kx,f^ky)\big),
\end{equation}
then we have for all $n\in\mathbb{N}$:
\begin{equation}
    \Omega_n^c(x,y)\leq \frac{c}{\lambda-c}\cdot\omega\big(\lambda^nd(x,y)\big).
\end{equation}

\end{lemma}

\begin{proof}
Let $x,y\in W^s_{\epsilon}(x_0).$ By the concavity of $\omega$:

\begin{equation*}
    \frac{\omega\big(\lambda^nt\big)}{\lambda^nt}\geq\frac{\omega\big(\lambda^kt\big)}{\lambda^kt}, ~\forall n,k,t\in\mathbb{R}^+\text{ and } k\leq n.
\end{equation*}
Now using the fact that $d(f^kx,f^ky)\leq\lambda^kd(x,y)$ for all $k\in\mathbb{N},$ we get:
\begin{align*}
    \Omega_n^c(x,y)&\leq\sum\limits_{k=0}^{n-1}c^{n-k}\omega\big(\lambda^kd(x,y)\big)\\
    &\leq\sum\limits_{k=0}^{n-1}\big(\frac{c}{\lambda}\big)^{n-k}\omega\big(\lambda^nt\big)\leq \frac{c}{\lambda-c}\omega\big(\lambda^nt\big).
\end{align*}

\end{proof}

\begin{lemma}\label{Regularity of unstable manifolds}
Let $f:\Lambda\to\Lambda$ be a $C^{1,\omega}$ hyperbolic map, then local unstable manifolds are $C^{1,\omega}.$
\end{lemma}

\begin{proof}
Consider $E$ (resp $F$) a smooth distribution close to $E^u$ (resp $E^s$) over $\Lambda,$ then consider the distribution $E^n:=f^n_*E.$ By definition $E^n$ converges exponentially to $E^u$ with respect to the distance $d=d_{E\oplus F}.$ So in order to prove the lemma, it is sufficient to prove that there is $C>0$ such that for all $n\geq 0$ the distribution $E^n$ has $C\omega$ as a modulus of continuity when restricted to an unstable leaf.

Fix $x_0\in \Lambda,$ and take $x,y\in W^u_\epsilon(x_0).$ Let $d_k=d_{E^u_{f^{-k}x_0}\oplus E^s_{f^{-k}x_0}}$ defined over distribution of dimension $q$ in a neighborhood of $f^{-k}x_0,$  we have:
\begin{align*}
d\big(E^{n+1}(x),E^{n+1}(y)\big)&=d\Big(df_{f^{-1}x}\big(E^n(f^{-1}x)\big),df_{f^{-1}y}\big(E^n(f^{-1}y)\big) \Big)\\
&\leq d\Big(df_{f^{-1}x}\big(E^n(f^{-1}x)\big),df_{f^{-1}x}\big(E^n(f^{-1}y)\big) \Big)\\&+d\Big(df_{f^{-1}x}\big(E^n(f^{-1}y)\big),df_{f^{-1}y}\big(E^n(f^{-1}y)\big) \Big)
\end{align*}

Upper bound for $d\Big(df_{f^{-1}x}\big(E^n(f^{-1}y)\big),df_{f^{-1}y}\big(E^n(f^{-1}y)\big) \Big):$
Since $f$ is $C^{1,\omega},$ we have $d\Big(df_{f^{-1}x}\big(E^n(f^{-1}y)\big),df_{f^{-1}y}\big(E^n(f^{-1}y)\big) \Big)\leq \omega\big(d(f^{-1}x,f^{-1}y)\big).$

Upper bound for $d\Big(df_{f^{-1}x}\big(E^n(f^{-1}x)\big),df_{f^{-1}x}\big(E^n(f^{-1}y)\big) \Big):$ using lemma \ref{Benoit approximation} and lemma \ref{Contraction of the action of hyperbolic maps} we deduce that
\begin{equation}
    (\delta\lambda)^2 d\big(E^n(f^{-1}x),E^n(f^{-1}y)\big)
\end{equation}
is the desired upper bound.

By induction, we deduce that for all $n\geq 0$ and $x,y\in W^u_\epsilon(x_0)$ we have:
\begin{align*}
d\big(E^n(x),E^n(y)\big)&\leq (\delta\lambda)^{2n}d\big(E(f^{-n}x),E(f^{-n}y)\big)+\sum\limits_{k=1}^n(\delta\lambda)^k\omega\big(\lambda^kd(x,y)\big)\\
&\leq d(x,y)+\big(\sum\limits_{k=1}^{+\infty}(\delta\lambda)^k\big)\omega\big(d(x,y)\big)\\
&\leq C\omega\big(d(x,y)\big),
\end{align*}
where the constant $C=C(\epsilon,\delta)$ does not depend on $x,y$ and $n.$ Finally we deduce that for all $x,y\in W^u_\epsilon(x_0)$ we have:
\begin{equation*}
    d(E^u_x,E^u_y)\leq C\omega\big(d(x,y)\big).
\end{equation*}

\end{proof}

\begin{lemma}\label{split regularity}
Let $X$ be a metric space, and $g:M\to X$ a continuous map. Assume that there is an $\epsilon,$ such that for all $x_0\in M$
\begin{align*}
    &d(gx,gy)\leq\omega\big(d(x,y)\big),\quad \forall x,y\in W_{\epsilon}^u(x_0),\\
   & d(gx,gy) \leq \omega\big(d(x,y)\big),\quad \forall x,y\in W_\epsilon^s(x_0).
\end{align*}
Then there is $K=K(\epsilon)>0$ such that for all $x,y \in M$ and $d(x,y)<\epsilon$ 
\begin{equation*}
    d(gx,gy)\leq 2\omega\big(Kd(x,y)\big).
\end{equation*}
\end{lemma}

\begin{proof}
For $\epsilon>0$ small, there is $K_1>0$ such that for all $x,y\in M$ with $d(x,y)<\epsilon$, we have the following inequality:
\begin{equation}\label{Pythagore}
    d(x,[x,y])^2+d([x,y],y)^2\leq K_1d(x,y)^2,
\end{equation}
where $[x,y]=W_\epsilon^s(x)\cap W_\epsilon^u(y)$ is the Bowen bracket of $x$ and $y$. Since $\omega$ is a modulus of continuity,
\begin{equation}
    \omega(s)+\omega(t)\leq 2\omega\Big(\frac{1}{\sqrt{2}}\sqrt{s^2+t^2}\Big), \forall s,t\geq 0,
\end{equation}
the last two inequalities implies that for all $x,y\in M$ satisfying $d(x,y)<\epsilon,$ 
\begin{align*}
    d(gx,gy)&\leq d(gx,g[x,y])+d(g[x,y],gy)\\
    &\leq \omega\big(d(x,[x,y])\big)+\omega\big(d([x,y],y)\big)\\
    &\leq 2\omega\big(\frac{1}{\sqrt{2}}\sqrt{d(x,[x,y])^2+d([x,y],y)^2}\big)\\
    &\leq 2\omega\big(\frac{1}{\sqrt{2}}\sqrt{K_1}d(x,y)\big).
\end{align*}
\end{proof}

\begin{proof}[Proof of theorem \ref{main theorem}]

Let $x_0\in\Lambda\subset M,$ and consider the positive orbit of $x_0.$ Take a neighborhood of $f^k(x_0)$ and $\psi_k:U_k\to\mathbb{R}^n$ a smooth chart. We can assume that for all $k\in\mathbb{N},$ $U_k\supset B(f^kx_0,\epsilon_0)$ for some $\epsilon_0>0,$ by taking $\epsilon_0$ less than the Lebesgue number. We can assume further that $d_{f^kx_0}\psi_k(E^u_{f^kx_0})=\mathbb{R}^q\times\lbrace0\rbrace$ and  $d_{f^kx_0}\psi_k(E^s_{f^kx_0})=\lbrace0\rbrace\times\mathbb{R}^{n-q}$. Let $g_k=\psi_{k+1}\circ f\circ\psi^{-1}_k$. We can choose $\epsilon>0$ small such that for all $k\in\mathbb{N},$ the map 
$$ g_k\circ g_{k-1}\circ\cdots\circ g_0:\psi_0(W^s_{\epsilon}(x_0))\to\mathbb{R}^n,$$ is well defined.

For $x\in W_g^s(x_0):=\psi_0\big(W^s_\epsilon(x_0)\big),$ let $E^u_x:=d\psi_0\big(\psi_0^{-1}(x)\big)\cdot E^u_{\psi_0^{-1}(x)}.$ To prove the theorem it will be enough to prove that:
\begin{align*}
W_g^s(x_0)&\to G^q(\mathbb{R}^n)\\
x&\mapsto E^u_x,    
\end{align*}
has a summable modulus of continuity.
By choosing $\epsilon >0$ smaller, we may assume that $E^u_x\oplus\mathbb{R}^{n-q}=\mathbb{R}^n, \forall x\in W_g^s(x_0).$

Let $x,y\in W^s_g(x_0), A_k=dg_k({g_{k-1}x})$ and $B_k=dg_k({g_{k-1}y})$ $\forall k\geq 0,$ (where $g_{-1}:=Id).$
Consider $d=d_{\mathbb{R}^q\oplus\mathbb{R}^{n-q}}$ then we have:
\begin{equation*}
    d(A_0E^u_x,B_0E^u_y)\leq d(A_0E^u_x,A_0E^u_y)+d(A_0E^u_y,B_0E^u_y).
\end{equation*}

Using lemma \ref{Benoit approximation}, we can find $\delta=\delta\big(W^s_{\epsilon}(x_0)\big)$ close to 1, such that 
\begin{equation*}
d(A_0E^u_x,B_0E^u_y) \leq\delta\cdot d_{E^u_x\oplus E^s_x}(A_0E^u_x,B_0E^u_y),     
\end{equation*} 

And using lemma $\ref{Contraction of the action of hyperbolic maps}$, we can find $\lambda=\lambda(W^s_{\epsilon}(x_0))\in (0,1)$ such that
\begin{equation*}
    d_{E^u_x\oplus E^s_x}(A_0E^u_x,A_0E^u_y)\leq \lambda^2d_{E^u_x\oplus E^s_x}(E^u_x,E^u_y),
\end{equation*}
so we deduce that 
\begin{equation}\label{First inequality}
d(A_0E^u_x,A_0E^u_y)\leq(\delta\lambda)^2\cdot d(E^u_x,E^u_y).    
\end{equation}

Since $f\in C^{1,\omega}$ we have $\|A_0-B_0\|\leq \omega(d(x,y)),$ then we apply lemma \ref{Continuity of the nice distance on distribution} and get 
\begin{equation*}
    d(A_0E^u_y,B_0E^u_y)\leq\omega(d(x,y)).
\end{equation*}
We deduce that for all $n\in\mathbb{N}$
\begin{align*}
    d(A_n&\cdots A_0E^u_x,B_n\cdots B_0E^u_y)\\&\leq (\delta\lambda)^2d(A_{n-1}\cdots A_0E^u_x,B_{n-1}\cdots B_0E^u_y)+\omega\big( d(g_{n-1}\cdots g_0x,g_{n-1}\cdots g_0y)\big)    \end{align*}

Using induction
\begin{align*}
d(A_n&\cdots A_0E^u_x,B_n\cdots B_0E^u_y) \\&\leq (\delta\lambda)^{2n}\cdot d(E_x^u,E^u_y)+\sum\limits_{k=0}^{n-1}(\delta^2\lambda^2)^{n-k}\omega\big(d(g_k\cdots g_0x,g_k\cdots g_0y)\big). \end{align*}
Since $\delta$ is arbitrarily close to $1$, we can assume that $(\delta\lambda)\in(0,1)$. This implies using lemma \ref{Summable modulus majoration} that
\begin{equation}
d(A_n\cdots A_0E^u_x,B_n\cdots B_0E^u_y) \leq (\delta\lambda)^{2n}\cdot d(E_x^u,E^u_y)+\frac{\delta^2\lambda}{1-\delta^2\lambda}\omega\big(\lambda^nd(x,y)\big). 
\end{equation}
Since these inequalities do not depend on charts up to multiplication by constant, we deduce that there is $M_1,M_2,M_3>0$ such that for all $ x,y\in W^s_{\epsilon}(x_0)$ and $n\in\mathbb{N}$
\begin{equation}\label{main inequality}
    d(E^u_{f^nx},E^u_{f^ny})\leq M_1(\delta\lambda)^{2n}d(E^u_x,E^u_y)+M_2\omega\big(M_3\lambda^nd(x,y)\big).  
\end{equation}

Let $\omega^u$ be the modulus of continuity of $E^u$ restricted to $W^s_\epsilon(x_0).$ Using (\ref{main inequality}) we deduce that
\begin{equation}
    \omega^u(\lambda_0^nt)\leq M_1(\delta\lambda)^{2n}\omega^u(t)+M_2M_3\omega\big(\lambda^nt\big), \forall n\in\mathbb{N}, t>0,
\end{equation}
where $\lambda_0\in(0,\lambda)$ only depends on the hyperbolic map $f.$ The latter inequality implies that $\omega^u$ is Dini summable. Using lemmas \ref{Regularity of unstable manifolds} and \ref{split regularity} we deduce that $E^u$ has a summable modulus of continuity, hence $\psi^{(u)}$ has a summable modulus of continuity.
\end{proof}

The next lemma is proven in Chapter 4, Lemma 4.3 of \cite{bowen1975equilibrium}. Together with the regularity of the geometric potential, it implies the existence and uniqueness of an equilibrium measure.

\begin{lemma}\label{unicity of equilibrium measure}
Let $\sigma:\Sigma_A\to\Sigma_A$ a subshift of finite type, and $f:\Lambda\to\Lambda$ semiconjugated to $\sigma$ via $\pi,$ then for any $\mu\in\mathcal{M}_f(\Lambda)$ there is a $\nu\in\mathcal{M}_{\sigma}(\Sigma_A)$ with $\pi_*\nu=\mu.$
\end{lemma}

\begin{proof}[Proof of corollary 1]

Using Markov partition, the map $f$ is semiconjugated to a subshift of finite type \cite{bowen1975equilibrium}. Let us consider the following commutative diagram: 
\begin{center}
\begin{tikzpicture}
  \matrix(m)[matrix of math nodes,row sep=2em,column sep=4em,minimum width=2em]
  { \Sigma_A & \Sigma_A &\\
    \Lambda & \Lambda & \mathbb{R} \\};
  \path[-stealth]
    (m-1-1) edge node [above] {$\sigma$} 
    (m-1-2) edge node [left] {$\pi$} (m-2-1)
    (m-1-2) edge node [right] {$\pi$} (m-2-2)
    (m-2-1) edge node [above] {$f$} (m-2-2)
    (m-1-2) edge node [above] {$\null\kern4ex\phi^u\circ\pi$} (m-2-3)
    (m-2-2) edge node [above] {$\phi^u$} (m-2-3);
\end{tikzpicture}
\end{center}
Since $\phi^u$ is $C^{0,\omega_0}$, where $\omega_0$ is a modulus that satisfies Dini condition (theorem \ref{main theorem}) and $\pi$ is Hölder continuous, then $\phi^u\circ\pi$ is $C^{0,\omega_0}$, so Theorem 1 of \cite{FanJiang} give us a unique equilibrium state $\mu_{\phi^u\circ\pi}$ for $(\sigma,\phi^u\circ\pi).$ We push this measure by $\pi$ to get a measure $\mu_{\phi^u},$ and using the fact that $\pi$ is a bijection when restricted to a set of full $\nu$ measure, we deduce that $h_{\mu_{\phi^u}}(f)=h_{\phi^u\circ\pi}(\sigma),$ which implies that $\mu_{\phi^u},$ is an equilibrium measure. If $\mu_2$ is another equilibrium measure, we lift it using the previous lemma to an equilibrium measure of $(\sigma,\phi^u\circ\pi),$ and by uniqueness of the equilibrium measure for the shift, we get $\mu_2=\mu_{\phi^u}.$

Since $\mu_{\phi^u}$ is the unique equilibrium state, it is ergodic. Indeed if
$$\mu_{\phi}=t\mu_1+(1-t)\mu_2,$$
then we have:

\begin{align*}
P(\phi^{(u)}) &= h_{\mu_{\phi^{(u)}}}(f)+\int \phi^{(u)}~d\mu    \\
&= th_{\mu_1}(f)+(1-t)h_{\mu_2}(f)+t\int \phi^u~d\mu_1+(1-t)\int \phi^u~d\mu_2\\ 
&\leq P(\phi^{(u)}),
\end{align*}
which implies that $P_{\mu_1}(\phi^u)=P_{\mu_2}(\phi^u)=P(\phi^u),$ so $\mu_1$ and $\mu_2$ are equilibrium states, this implies that $\mu=\mu_1=\mu_2.$

\end{proof}

\section{Proof of theorem \ref{Physic}}

Given Theorem \ref{main theorem} the proof of Theorem \ref{Physic} is in the same way as in \cite{bowen1975equilibrium}, but we need to adapt a few lemmas in low regularity.
The following three lemmas are crucial to prove that the equilibrium measure is a physical measure.

\begin{lemma}[Distortion lemma]\label{Distortion lemma}
Let $f:\Lambda\to\Lambda$ a $C^{1+\omega}$ hyperbolic map, where $\omega$ is a Dini summable modulus. Fix $\epsilon>0$ and an unstable invariant cone family $(C^u_x)_{x\in U}$ of sufficiently small angle,  and  $\mathcal{F}_{x,n}$ a foliation of $B_n(x,\epsilon)$ tangent to $C^u,$ whose leafs are $C^{1+\omega}$ and $E=T\mathcal{F}_{x,n}$ is $C^{0,\omega}$ distribution. Then there is $C=C(\epsilon)$ such that for all $n\in\mathbb{N}, x\in\Lambda$ and $y\in B_n(x,\epsilon)$
\begin{equation}\label{Distortion}
 \frac{1}{C}\leq \frac{\det df^n(x)_{|E(x)}}{\det df^n(y)_{|E(y)}}\leq C,
\end{equation}
We have also for all $n\in \mathbb{N}:$
\begin{equation}\label{Typical expansion for subspaces!}
\frac{1}{C}\leq \frac{\det df^n(x)_{|E(x)}}{J^u f^n(x)}\leq C.    
\end{equation}
\end{lemma}

\begin{remark}
If $f$ is an Anosov diffeomorphism the foliation $\big((W^u_\epsilon(y)\big)_{y\in W^s_\epsilon(x)}$ satisfies the previous lemma.
\end{remark}

\begin{proof}
Let $x\in \Lambda, z\in W^s_\epsilon(x)$ and $y\in P(z,\epsilon),$ where $P(z,\epsilon)$ is a leaf of $\mathcal{F}_{x,n}$ passing through $z.$
By the regularity of $E$ and of $f,$ there is $C_0=C_0(\epsilon)$ such that for all $n\in\mathbb{N}$ 
\begin{align*}
    \left\lvert\frac{\det df(f^nx)_{|df^n_xE(x)}}{\det df(f^nz)_{|df^n_zE(z)}}-1\right\rvert\leq C_0\Big(\alpha\big(df^n_xE(x)\big)+\alpha\big(df^n_yE(z)\big)+\omega\big(d(f^nx,f^nz)\big) \Big),
\end{align*}
where $\alpha(E(*))=d_{E^u(*)\oplus E^s(*)}\big( E^u(*),E(*)\big).$\\ 
Using lemma \ref{Contraction of the action of hyperbolic maps}, we have $\alpha\big(df^n_*E(*)\big)\leq \lambda^{2n}\alpha\big(E(*)\big),$ so we deduce that

\begin{equation*}
  \left\lvert\frac{\det df(f^nx)_{|df^n_xE(x)}}{\det df(f^nz)_{|df^n_zE(z)}}-1\right\rvert\leq C_0\Big(\lambda^{2n}\alpha\big(E(x)\big)+\lambda^{2n}\alpha\big(E(z)\big)+\omega\big(\lambda^nd(x,z)\big)  \Big),
\end{equation*}
in particular we have
\begin{align*}
    \frac{\det df^n(x)_{|E(x)}}{\det df^n(z)_{|E(z)}}&=\prod\limits_{k=0}^{n-1}\frac{\det df(f^kx)_{|df^k_xE(x)}}{\det df(f^kz)_{|df^k_zE(z)}}\\
    &\leq \prod\limits_{k=0}^{n-1}\Big(1+\lambda^{2k}+\omega(\lambda^k\epsilon)\Big)\leq C_1=C_1(\epsilon).
\end{align*}

Now, $f^k\big(P(z,\epsilon)\big)$ has diameter of order $\lambda^{n-k}$, 
so
\begin{equation*}
\frac{\det df^n(z)_{|E(z)}}{\det df^n(y)_{|E(y)}}=\prod\limits_{k=0}^{n-1}\frac{\det df^k(z)_{|df_z^kE(z)}}{\det df^k(x)_{|df_y^kE(x)}}\leq \prod\limits_{k=0}^{n-1}\big(1+\omega(\lambda^{n-k}\epsilon)\big)\leq C_2=C_2(\epsilon).
\end{equation*}
This proves the first part of the lemma. To prove the second part we use the fact that the action of $Df_x$ on $G_x^q$ 
is contracting ( lemma \ref{Contraction of the action of hyperbolic maps}). Indeed, we have for all $k\in\mathbb{N}:$
\begin{equation*}
    \frac{\det df^k(x)_{|E(x)}}{J^uf^k_x}=\prod\limits_{i=0}^{k-1}\frac{\det df(f^ix)_{|df_x^iE(x)}}{J^uf(f^ix)}\leq\prod\limits_{i=0}^{+\infty}\big(1+\lambda^{2i}d(E_x^u,E_x)\big)\leq C_3=C_3(\epsilon,E_x),
\end{equation*}
which finishes the proof of the second claim.
\end{proof}

\begin{lemma}[Bowen-Ruelle \cite{BowenRuelleAxiomAFlows}]\label{First volume lemma}
Let $\Lambda$ be an attractor of class $C^{1,\omega}$, where $\omega$ is a Dini summable modulus. Let $\epsilon>0,$ then there exist $C=C(\epsilon)$ such that for all $x\in\Lambda$ and $n\in\mathbb{N}$
\begin{equation}
  C^{-1}\cdot\frac{1}{J^uf^n(x)} \leq  vol^m\big(B_n(x,\epsilon)\big)\leq C\cdot\frac{1}{J^uf^n(x)},
\end{equation}

\end{lemma}

\begin{proof}
As in \cite{BowenRuelleAxiomAFlows}, consider for each $z\in\Lambda$ a local chart $\phi_z:T_zM(\epsilon)\to M$ such that $\phi_z\big(E^u_z)\big)\subset W^u(z)$ and $\phi_z\big(E^s_\epsilon(z)\big)\subset W^s(z)$ ($E^s_\epsilon(z)$ is the open ball whose center is the origin of $E^s_z$ and radius $\epsilon$) and such that the maps $F_z=\phi^{-1}_{fz}\circ f\circ\phi_z$ is tangent to $D_zf$ at the origin of $T_zM.$

For $x\in\Lambda,$ and $n\in\mathbb{N}$ consider the set:
\begin{equation*}
D_n(x,\epsilon)=\left\lbrace u\in T_xM:~ \|F^ku\|_{f^kx}\leq \epsilon,\text{  for } k=0,\cdots,n-1 \right\rbrace,    
\end{equation*}
then by definition of $D_n(x,\epsilon),$ there are $C_1,C_2>0$ independent of $n$ and such that
\begin{equation*}
    \phi_x\big(D_n(x,C_1\epsilon)\big)\subset B_n(x,\epsilon)\subset \phi_x\big( D_n(x,C_2\epsilon)\big).
\end{equation*}
So to estimate the volume of $B_n(x,\epsilon),$ it will be enough to estimate the volume of $D_n(x,\epsilon).$

For $u\in T_xM,$ let $(u_1,u_2)$ be the decomposition of $u$ with respect to the splitting $T_xM=E^u_x\oplus E^s_x$, then consider $v\in E^s_\epsilon(x)$ and define the set
\begin{equation*}
P_n(v,\epsilon)=\left\lbrace u\in T_xM=E^u_x\oplus E^s_x: u_2=v, (F^ku)_1\in E^u_\epsilon(f^kx) \text{ for } 0\leq k\leq n-1   \right\rbrace.
\end{equation*}

Let $K=K(\epsilon)>0$ such that for all $k\geq 0$ we have $\|F^kv\|\leq (K+1)\epsilon.$

\begin{fact}
For $\epsilon>0$ sufficiently small, there is $\gamma>0$ such that for all $x\in\Lambda, n\in\mathbb{N}$ and $k\in\lbrace 0,\cdots,n-1\rbrace,$ the set $F^k\big(P_v(\epsilon,n)\big)$ is a graph of a $C^1$ function $\psi_k:E^u_\epsilon(f^kx)\to E^s_{(K+2)\epsilon}(f^kx)$ such that $\|D\psi_k\|\leq \gamma.$
\end{fact}
Once we prove this fact, we deduce that:
\begin{equation}\label{foliating the dynamical ball}
    D_n(x,\epsilon)\subset \bigcup\limits_{v\in E^s_\epsilon(x)}P_\epsilon(v,n)\subset D_n\big(x,(K+3)\epsilon\big),
\end{equation}
so estimating the size of $P_\epsilon(v,n)$ together with the distortion lemma (lemma \ref{Distortion lemma}) finishes the proof. Indeed, using \eqref{foliating the dynamical ball} we get:
\begin{equation*}
    vol\big(D_n(x,\epsilon)\big)\leq vol\big(\bigcup\limits_{v\in E^s_\epsilon(x)}P_\epsilon(v,n)\big)\leq vol\Big(D_n\big(x,(K+3)\epsilon\big)\Big),
\end{equation*}
and by Fubini we get:
\begin{equation*}
    vol\big(\bigcup\limits_{v\in E^s_\epsilon(x)}P_\epsilon(v,n)\big)=\int_{E_\epsilon(x)}vol^q\big(P_n(v,\epsilon)\big)~d~vol^q(v),
\end{equation*}
then the distortion lemma implies that:
\begin{equation*}
    \frac{C^{-2}}{J^uf^n(x)}\leq vol^q\big(P_n(v,\epsilon)\big)\leq \frac{C^2}{J^uf^n(x)}.
\end{equation*}
\begin{proof}[Proof of Fact 1]
Fix $k\in\lbrace 0,\cdots,n-1\rbrace,$ and let $W$ be a graph of a $C^1$ map $\varphi:E^u_\epsilon(f^kx)\to E^s_{(K+2)\epsilon}(f^kx)$ with a Lipschitz constant $\gamma\leq 1$.

It is clear that for $\epsilon>0$ small enough, $F_*($Graph $\varphi)$ is a graph of $C^1$ function $\Tilde{\varphi}.$

If $u=(u_1,u_2) \in D_{n-k}(f^kx,\epsilon)$, write $F$ as:
\begin{equation}
    F(u_1,u_2)=\big(\Tilde{F}u_1+\alpha(u_1,u_2),\Tilde{F}u_2+\beta(u_1,u_2)\big),
\end{equation}
where $\Tilde{F}$ is $DF$ at the origin of $T_{f^kx}M,$ and $\|\alpha\|_{C^1},\|\beta\|_{C^1}<\delta(\epsilon),$ and $\delta(\epsilon)\to 0$ when $\epsilon\to 0.$

Let $(u_1',w_1'), (u_2',w_2')\in$ Graph $(\Tilde{\varphi}),$ then there is a point $(u_i,w_i)\in$ Graph $(\varphi)$ such that $F(u_i,w_i)=(u_i',w_i').$ So we deduce the following:
\begin{align*}
    \|w_2'-w_1'\|&=\|\Tilde{F}(w_2-w_1)+\beta(u_2,w_2)-\beta(u_1,w_1)\|\\
    &\leq \lambda\|w_2-w_1\|+\delta(\gamma+1)\|u_2-u_1\|\\
    &\leq \lambda\gamma\|u_2-u_1\|+\delta(\gamma+1)\|u_2-u_1\|\\
    &=\big(\lambda\gamma+\delta(\gamma+1)\big)\|u_2-u_1\|,
\end{align*}
we have also:
\begin{align*}
    \|u_2'-u_1'\|&=\|\Tilde{F}(u_2-u_1)+\alpha(u_2,w_2)-\alpha(u_1,w_1)\|\\
    &\geq \frac{1}{\lambda}\|u_2-u_1\|-\delta(\gamma+1)\|u_2-u_1\|\\
    &=\big(\frac{1}{\lambda}-\delta(\gamma+1)\big)\|u_2-u_1\|,
\end{align*}
so by choosing $\epsilon$ small enough we can take any $\gamma\leq 1,$ which finish the proof of the fact.
\end{proof}\end{proof}

The following lemma is a variation of the previous lemma. It provides a lower bound of the volume of a dynamical ball centered near the hyperbolic attractor, this is crucial to find a link between being an attractor and having zero pressure with respect to the geometric potential.  \cite{BowenRuelleAxiomAFlows}

\begin{lemma}[Bowen-Ruelle \cite{BowenRuelleAxiomAFlows}]\label{Second volume lemma}
For all small $\epsilon, \delta>0$ there is $d=d(\epsilon,\delta)>0$ such that for all $n\in\mathbb{N}, x\in\Lambda$ and $y\in B_n(x,\epsilon)$ we have::
$$vol^m\big(B_n(y,\delta)\big)\geq d\cdot vol^m\big(B_n(x,\epsilon)\big).$$

\end{lemma}

\begin{proof}
If $y\in\Lambda,$ then the inequality of the lemma is obvious by the previous lemma. Assume that $y\notin\Lambda.$ Since $W^s_\epsilon(\Lambda)$ is a neighborhood of $\Lambda$ (because $\Lambda$ is an attractor) there is $z\in\Lambda$ such that $y\in W^s_\epsilon(z),$ and since $y\in B_n(x,\epsilon)$ we have $z\in B_n(x,2\epsilon).$ Let $A=[x,z]=W^u_\epsilon(x)\cap W^s_\epsilon(z).$ By the shadowing lemma and the expansiveness of $f$ in $\Lambda$ the point $A$ belongs to $\Lambda.$

By construction, $A\in W^u_\epsilon(x)\cap B_n(x,3\epsilon),$ so by the previous lemma the volume of $B_n(x,\epsilon)$ and $B_n(A,\delta)$ are proportional independently of $n,$ so in order to prove this lemma, it will be enough to compare the volume of $B_n(y,\delta)$ and the volume of $B_n(A,\epsilon).$

By this remark, we may assume in that $y\in W^s_\epsilon(x).$ Using the same argument as in the previous lemma, we prove similarly that $\bigcup\limits_{v\in W^s_\delta(y)}P_n(y,\delta)\subset D_n(y,\delta),$ then using distortion lemma, we get the desired inequality.

\end{proof}


The following lemma is proven in Chapter 4 Lemma 4.9 of \cite{bowen1975equilibrium}.
\begin{lemma}\label{relation unstale piece Attractor}
Let $\Lambda$ be a hyperbolic set of a $C^1$ diffeomorphism. If $W^u_\epsilon(x)\subset \Lambda$ for some $x,$ then $\Lambda$ is an attractor. If $\Lambda$ is not an attractor, then there exists $\gamma>0$ such that for every $x\in\Lambda,$ there is $y\in W^u_\epsilon(x)$ with $d(y,\Lambda)\geq \gamma.$
\end{lemma}

The proof of the following proposition and corollary 2 is the same as Theorems 4.11 and 4.12 of \cite{bowen1975equilibrium}. For convenience, we sketch the proofs.
\begin{proposition}\label{Attractor different formulation}
Let $f:\Lambda\to \Lambda$  be a transitive hyperbolic map of class $C^{1,\omega}$, where $\omega$ is Dini summable modulus, then the following are equivalent:
\begin{itemize}

    \item[(i)]  $\Lambda$ is an attractor.
    
    \item[(ii)]  $vol^m\big(W^s(\Lambda)\big)>0.$
    
    \item[(iii)] $P_{f_{|\Lambda}}(\phi^{(u)})=0.$
    
\end{itemize}
\end{proposition}

\begin{proof}

$(i) \Rightarrow (ii)$ This implication is in fact true if $f$ is only $C^1$. Indeed we have $W^s(\Lambda)=\bigcup\limits_{x\in\Lambda}W^s(x),$ which implies that $W^s(\Lambda)$ is a neighborhood of $\Lambda.$

$(ii) \Rightarrow (iii)$ Define $s(\epsilon,n)$ by:
\begin{equation*}
    s(\epsilon,n)=\sup\limits_{\mathcal{S}\in\mathcal{S}_{\epsilon,n}}\sum\limits_{x\in\mathcal{S}}e^{S_n\phi^{(u)}(x)}=\sup\limits_{\mathcal{S}\in\mathcal{S}_{\epsilon,n}}\sum\limits_{x\in\mathcal{S}}\frac{1}{J^uf(x)},
\end{equation*}
where $S_{\epsilon,n}$ is the set of $(\epsilon,n)-$separated sets of $\Lambda.$ Using lemma \ref{First volume lemma} we have for all $\mathcal{S}\in S_{\epsilon,n}$: 
\begin{align*}
    s(\epsilon,n) &\geq C_{\epsilon}^{-1}\sum\limits_{x\in\mathcal{S}}vol^m\big(B_n(x,\epsilon)\big)\\
    &\geq C_\epsilon^{-1} vol^m\big(\bigcup\limits_{x\in\mathcal{S}}B_n(x,\epsilon)\big)\\
    &\geq C_\epsilon^{-1} vol^m\big(W^s_{\epsilon/2}(\Lambda)\big),
\end{align*}
which implies that:
\begin{equation*}
    P_{f|\Lambda}(\phi^{(u)})=\lim\limits_{\epsilon\to 0}\lim\limits_{n\to +\infty}\frac{1}{n}\log s(\epsilon,n)\geq 0
\end{equation*}
Similarly we have:
\begin{align*}
    s(\epsilon,n)&\leq C_\epsilon\sum\limits_{x\in\mathcal{S}}vol^m\big(B_n(x,\epsilon)\big)\label{Technical stuff to prove P<0}\\
    &\leq C_\epsilon C_{\epsilon/2} \sum\limits_{x\in\mathcal{S}}vol^m\big(B_n(x,\epsilon/2)\big)\\
    &\leq C_\epsilon C_{\epsilon/2} vol^m\big(\bigcup\limits_{x\in\mathcal{S}}B_n(x,\epsilon/2)\big)\\
    &\leq C_\epsilon C_{\epsilon/2} vol^m\big(W^s_\epsilon(\Lambda))
\end{align*}
where $\mathcal{S}$ is $(\epsilon,n)-$separated set, and the third inequality follows from the fact that $B_n(x,\epsilon/2)$ where $x$ varies in $\mathcal{S}$ are disjoint. Hence we deduce that \begin{equation*}
    P_{f|\Lambda}(\phi^{(u)})\leq 0.
\end{equation*}

$(iii) \Rightarrow (i)$ Assume that $\Lambda$ is not an attractor. We will prove that the pressure is negative. Let $\epsilon>0$ small, and choose $\gamma>0$ as in lemma \ref{relation unstale piece Attractor}. Let $N\in\mathbb{N}$ such that 
$$W^u_\epsilon(f^Nx)\subset f^N(W^u_{\gamma/4}), \forall x\in\Lambda.$$
Let $\mathcal{S}\subset \Lambda$ be $(\gamma,n)-$separated. Using lemma \ref{relation unstale piece Attractor}, there is a point $y(x,n)\in B_n(x,\gamma/4)$ such that
$$d(f^{n+N}y(x,n),\Lambda)>\gamma.$$
Choose $\delta\in(0,\gamma/4)$ so that $d(f^Nz,f^Ny)<\gamma/2$ whenever $d(z,y)<\delta.$ Then $B_n(y(x,n),\delta)\subset B_n(x,\gamma/2),$ and $f^{n+N}B_n(y(x,n),\delta)\cap B(\Lambda,\gamma/2)=\emptyset.$

So we deduce that $B_n(y(x,n),\delta)\cap B_{n+N}(\mathcal{S},\gamma/2)=\emptyset.$ Using lemma \ref{Second volume lemma} we get
\begin{align*}
    vol^m\big(B_n(\mathcal{S},\gamma/2)\big)-vol^m(B_{n+N}\big(\mathcal{S},\gamma/2)\big)&\geq\sum\limits_{x\in\mathcal{S}}vol^m(B_n(y(x,n),\delta))\\
    &\geq d(3\gamma/2,\delta)\sum\limits_{x\in\mathcal{S}}vol^m(B_n(x,3\gamma/2))\\
    &\geq d(3\gamma/2,\delta)vol^m(B_n(\mathcal{S},\gamma/2),
\end{align*}
so we get for all $n>N$:
$$vol^m(B_{n+N}(\mathcal{S},\gamma/2))\leq (1-d)vol^m(B_n(\mathcal{S},\gamma/2)),$$
finally, using the upper bound of $s(n,\epsilon)$, we deduce that: $$P_{f|\Lambda}(\phi^{(u)})\leq \frac{1}{N}\log(1-d)<0.$$
\end{proof}

The following lemma is proven in \cite{bowen1975equilibrium}.
\begin{lemma}\label{Quasi-Gibbs property}
Let $\phi:\Lambda\to\mathbb{R}$ be a $C^{0,\omega}$ potential, where $\omega$ is Dini summable, and $P=P_{f|\Lambda}(\phi)$ the pressure of $f$ restricted to $\Lambda.$ Then for small $\epsilon>0$ there is $b_\epsilon>0$ such that for any $x\in\Lambda$ and $n\in\mathbb{N}$ we have:
\begin{equation}\label{Quasi Gibbs property-nequality}
\mu_{\phi}(B_n(x,\epsilon))\geq b_\epsilon\exp(-Pn+S_n\phi(x)).    
\end{equation}
\end{lemma}

\begin{proof}[Proof of theorem \ref{Physic}:]
Let $g:U\to \mathbb{R}$ be a continuous function. Put $\bar{g}(n,x)=\frac{1}{n}\sum\limits_{k=0}^{n-1}g(f^kx)$ and $\bar{g}=\int g~d\mu_{\phi^{(u)}}.$ Fix a small $\delta>0,$ and consider the sets
\begin{equation*}
    C_n(g,\delta)=\lbrace x\in M~|~|\bar{g}(n,x)-\bar{g}|>\delta\rbrace,\\
   \quad B(g,\delta)=\bigcap\limits_{N=1}^{\infty}\bigcup\limits_{n=N}^{\infty}C_n(g,\delta).
\end{equation*}
We want to prove that for all $\delta>0$ the volume of $B(g,\delta)$ is zero, which proves the physicality.
Take $\epsilon>0,$ such that we have $d(gx,gy)<\delta$ whenever $d(x,y)<\epsilon,$ then fix $N$ and choose $\mathcal{R}_N,\mathcal{R}_{N+1},\cdots$ as follow:
$\mathcal{R}_n$ is a maximal subset of $\Lambda\cap C_n(g,2\delta),$ satisfying:
\begin{itemize}
    \item $B_n(x,\epsilon)\cap B_k(y,\epsilon)=\emptyset$ for $x\in\mathcal{R}_n,y\in\mathcal{R}_k$, and $N\leq k<n,$
    \item $B_n(x,\epsilon)\cap B_n(x',\epsilon)=\emptyset$ for $x,x'\in\mathcal{R}_n$ and $x\neq x'.$
\end{itemize}

Let $V_N=\bigcup\limits_{k=N}^{\infty}\bigcup\limits_{x\in\mathcal{R}_k} B_k(x,\epsilon),$ which is a disjoint union by definition of $(\mathcal{R}_n)_n.$ We have $B_k(x,\epsilon)\subset C_k(g,\delta)$ so $V_N\subset\bigcup\limits_{k=N}^{\infty}C_k(g,\delta).$

Since $\mu_{\phi^{(u)}}$ is ergodic, we have $$0=\mu_{\phi^{(u)}}(B(g,\delta))=\lim\limits_{n\to\infty}\mu_{\phi^{(u)}}(\bigcup\limits_{n=N}^{\infty}C_n(g,\delta)),$$ which implies that
\begin{equation}\label{Birkhoof implies V_N goes to 0}
    \lim\limits_{N\to\infty}\mu_{\phi^{(u)}}(V_N)=0.
\end{equation}
So using the fact that $P_{f|\Lambda}(\phi^{(u)})=0$ and lemma $\ref{Quasi-Gibbs property}$ we get
\begin{equation}\label{Gibbs property to prove observability}
    \mu_{\phi^{(u)}}(V_N)\geq b_\epsilon\sum\limits_{k=N}^{\infty}\sum\limits_{x\in\mathcal{R}_k}\exp(S_k\phi^{(u)}(x)).
\end{equation}

Now for $x\in\Lambda$ and $y\in W^s_\epsilon(x)\cap C_n(g,3\delta)$ we have $x\in C_n(g,2\delta),$ so in particular:
\begin{equation*}
    W^s_\epsilon(\Lambda)\cap\bigcup\limits_{k=N}C_k(g,3\delta)\subset \bigcup\limits_{k=N}^{\infty}\bigcup\limits_{x\in\mathcal{R}_k}B_k(x,2\epsilon),
\end{equation*}
so using lemma \ref{First volume lemma} we deduce that:
\begin{equation}\label{second part to prove observability}
    vol^m(W^s_\epsilon(\Lambda)\cap\bigcup\limits_{k=N}^\infty C_n(g,3\delta))\leq C_{2\epsilon}\sum\limits_{k=N}^\infty\sum\limits_{x\in\mathcal{R}_k}\exp(S_k\phi^{(u)}(x)).
\end{equation}
Finally, using (\ref{Birkhoof implies V_N goes to 0}),(\ref{Gibbs property to prove observability}) and (\ref{second part to prove observability}) we deduce that $vol(B(g,\delta)\cap W^s_\epsilon)=0$ which ends the proof.
\end{proof}

\bibliographystyle{plain}
\bibliography{bibliography.bib}

\end{document}